\documentclass[reqno,a4paper]{amsart}
\usepackage{tikz,xcolor}

\usepackage[colorlinks=true, pdfstartview=FitV, linkcolor=blue, citecolor=blue, urlcolor=blue]{hyperref}

\newcommand{\ZZ}{\mathbb{Z}}
\newcommand{\QQ}{\mathbb{Q}}

\theoremstyle{plain}
\newtheorem{thm}{Theorem}[section]

\newtheorem{conj}[thm]{Conjecture}

\newtheorem{cor}[thm]{Corollary}
\theoremstyle{definition}

\newtheorem{ex}[thm]{Example}
\newtheorem{prob}[thm]{Problem}

\numberwithin{equation}{section}

\newcommand{\content}{c}
\newcommand{\hook}{h}
\newcommand{\EYD}{\mathcal{E}}
\newcommand{\SST}{\operatorname{SST}}
\newcommand{\abs}[1]{\left\lvert #1 \right\rvert}

\definecolor{darkred}{rgb}{0.7,0,0} 
\newcommand{\defn}[1]{{\color{darkred}\emph{#1}}} 

\usepackage[colorinlistoftodos]{todonotes}

\begin{document}
\title{Hook-content formula using excited Young diagrams}

\author[A.~Kirillov]{Anatol N. Kirillov}
\address[A.~Kirillov]{Research Institute for Mathematical Sciences (RIMS), 
Kyoto University, Kyoto 606-8502, Japan; ~the Kavli Institute for the Physics 
and Mathematics of the Universe (IPMU), 277-8583, Kashiwanoha, Japan; ~Department of Mathematics, National Research University Higher School of Economics (HES),7 Vavilova Str., 117312, Moscow, Russia}
\email{kirillov@kurims.kyoto-u.ac.jp }
\urladdr{http://www.kurims.kyoto-u.ac.jp/~kirillov/}

\author[T.~Scrimshaw]{Travis Scrimshaw}
\address[T.~Scrimshaw]{School of Mathematics and Physics, The University of Queensland, St.\ Lucia, QLD 4072, Australia}
\email{tcscrims@gmail.com}
\urladdr{https://people.smp.uq.edu.au/TravisScrimshaw/}


\thanks{TS was partially supported by the Australian Research Council DP170102648.}

\begin{abstract}
We construct a hook-content formula and its $q$-analog using excited Young diagrams analogous to Naruse's hook-length formula for skew shapes. Furthermore, we show that our hook-content formula has a simple factorization and give some conjectures and questions related to its $q$-analog.
\end{abstract}

\maketitle

\section{Introduction}

The hook-length formula for the number of standard Young tableaux of skew shape $\lambda/\mu$
\begin{equation}
\label{eq:hook_length_formula}
f^{\lambda/\mu} := \abs{\lambda/\mu} ! \sum_{D \in \EYD(\lambda/\mu)} \prod_{d \in \lambda \setminus D} \frac{1}{\hook(d)},
\end{equation}
where $\EYD(\lambda/\mu)$ is the set of excited Young diagrams~\cite{Kreiman05,IN09} and $h(d)$ is the hook length of $d$ in $\lambda$, was discovered by Naruse~\cite{Naruse14} from his study of the equivariant cohomology of the Grassmannian.
Combinatorial proofs of Equation~\eqref{eq:hook_length_formula} have also been given in~\cite{Konvalinka17,MPP18}.
When $\mu = \emptyset$, Equation~\eqref{eq:hook_length_formula} reduces to the classical hook-length formula for standard tableaux first proven by Frame, Robinson, and Thrall~\cite{FRT54} and has since seen numerous proofs (see, \textit{e.g.},~\cite{Bandlow08,MPP18,Sagan90} and references therein).

In~\cite{MPP18}, a $q$-analog of Equation~\eqref{eq:hook_length_formula} was given as
\begin{equation}
\label{eq:q_hook_length_formula}
s_{\lambda/\mu}(1, q, q^2, \ldots) = \sum_{D \in \EYD(\lambda/\mu)} \prod_{(i,j) \in \lambda \setminus D} \frac{q^{\lambda'_j - i}}{1 - q^{h(i,j)}},
\end{equation}
where the left hand side is the principal specialization of the (skew) Schur function and$\lambda'$ is the conjugate partition to $\lambda$.
When taking $\mu = \emptyset$, we obtain the $q$-analog of the hook-length formula due to Stanley~\cite{Stanley71}:
\begin{equation}
\label{eq:classical_q_hook}
s_{\lambda}(1, q, q^2, \ldots) = q^{b(\lambda)} \prod_{d \in \lambda/\mu} \frac{1}{1-q^{h(d)}},
\end{equation}
where $b(\lambda) = \sum_{i=1}^{\ell} (i-1) \lambda_i$.
After removing the $q^{b(\lambda)}$ factor, Equation~\eqref{eq:classical_q_hook} is equal to the number of reverse plane partitions graded by their size, where a combinatorial proof is given by the Hillman--Grassl correspondence~\cite{HG76}.

To count the number of semistandard Young tableaux of shape $\lambda$ and maximum entry $n$, we instead use the \defn{hook-content formula} with its natural $q$-analog given by
\begin{equation}
\label{eq:q_hook_content}
s_{\lambda}(1,q,\dotsc,q^{n-1}, 0, 0, \ldots) = q^{b(\lambda)} \prod_{d \in \lambda} \frac{[n+c(d)]_q}{[h(u)]_q},
\end{equation}
where $[x]_q = \frac{1-q^x}{1-q}$ is the natural $q$-analog of $x$ (see,\textit{e.g.},~\cite[Thm~7.21.2]{ECII}) and $c(d)$ is the content of $d$.
Indeed, we see that when taking the limit $q \to 1$, we obtain a formula for the number of semistandard Young tableaux of shape $\lambda$ and maximum entry $n$.

The goal of this note is examine a natural hook-content generalization of Naruse's hook-length formula by combining Equation~\eqref{eq:hook_length_formula} and Equation~\eqref{eq:q_hook_content}.
We show that the result has a simple factorization as a product of $q$-integers of binomials in $n$.
Our result gives rise to many interesting conjectures and questions related to our formula, the natural $q$-analog of $f^{\lambda/\mu}$, and results related to representation theory.
In particular, we note that our formula (when $q \to 1$) does not count the number of semistandard skew tableaux of shape $\lambda/\mu$. Thus, finding a combinatorial formula (in particular using excited Young diagrams) for the principal specializations of skew Schur functions
\[
s_{\lambda/\mu}(1,q,\dotsc,q^{n-1}, 0, 0, \ldots)
\] 
remains an open problem.
Yet, our results might aid in understanding the relationship between excited Young diagrams and the representation theory of the symmetric group $S_n$ and/or $\mathfrak{gl}_n$ as
\[
s_{\lambda/\mu} = \sum_{\nu} c^{\lambda}_{\mu,\nu} s_{\nu},
\qquad\qquad
f^{\lambda/\mu} = \sum_{\nu} c^{\lambda}_{\mu,\nu} f^{\nu},
\]
where $c^{\lambda}_{\mu,\nu}$ are the Littlewood--Richardson coefficients.

\subsection*{Acknowledgements}

AK grateful to the RIMS and the IPMU for fruitful atmosphere and conditions for research, and financial support.
TS would like to thank Kyoto University for its hospitality during his visit in March, 2019.
This work has also been supported by JSPS KAKENHI 1605057. 
This work benefited from computations using \textsc{SageMath}~\cite{sage,combinat}.

\section{Preliminaries}

A \defn{partition} is a weakly decreasing sequence of positive integers.
We equate a partition $\lambda = (\lambda_1, \lambda_2, \dotsc, \lambda_{\ell})$ with a set of \defn{cells} $\{(i,j) \mid 1 \leq j \leq \ell, 1 \leq i \leq \lambda_j\}$ via the Young diagram of $\lambda$.
We will consider our Young diagrams using English convention.
For a partition $\mu \subseteq \lambda$, we form the \defn{skew partition} $\lambda / \mu$ as the set of cells $\lambda \setminus \mu$.
More generally, we call any finite set of cells $D \subseteq \ZZ^2_{>0}$ a \defn{diagram}.
The \defn{size} of a diagram $\abs{D}$ is the number of cells in $D$.

Let $\lambda' = (\lambda'_1, \lambda'_2, \dotsc, \lambda'_m) = \{(j,i) \mid (i,j) \in \lambda\}$, where $m = \lambda_1$, be the conjugate partition to $\lambda$.
Let
\[
\content(d) := j - i,
\qquad\qquad
\hook(d) := \lambda_i - j + \lambda'_j - i + 1,
\]
be the \defn{content} and \defn{hook length}, respectively, of a cell $d \in \lambda$.
Recall that the content of a cell $d$ is the diagonal the cell lies on and the hook length is the number of boxes in the row and column to the right and below, respectively, $d$, including also $d$ (\textit{i.e.}, the size of the largest hook shape whose corner is at $d$).

Let $\lambda/\mu$ be a skew partition with $\abs{\lambda/\mu} = n$.
A \defn{standard tableau of (skew) shape $\lambda / \mu$} is a bijection $T \colon \lambda/\mu \to \{1, \dotsc, n\}$ such that every row (resp.\ column) is increasing when read left to right (resp.\ top to bottom).
Let $f^{\lambda/\mu}$ denote the number of standard tableau of shape $\lambda/\mu$.
A \defn{semistandard tableau of (skew) shape $\lambda / \mu$} is a function $T \colon \lambda/\mu \to \ZZ_{>0}$ such that rows are weakly increasing and columns are strictly increasing.
Let $\SST^n(\lambda/\mu)$ denote the set of semistandard Young tableaux of shape $\lambda/\mu$ with maximum entry $n$, and we simply write $\SST(\lambda/\mu)$ when $n = \infty$.
We will simply write $\lambda$ for $\lambda / \mu$ when $\mu = \emptyset$.

Following~\cite{IN09}, define an \defn{elementary excitation} on a diagram $D$ to take a cell $(i,j) \in D$ such that $(i+1,j), (i,j+1), (i+1,j+1) \notin D$ and forming a new diagram by $(D \setminus \{(i,j)\}) \cup \{(i+1,j+1)\}$.
Pictorially, an elementary excitation moves the cell in $(i,j)$ (locally) as
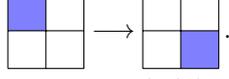
\[
\begin{tikzpicture}[scale=0.5,baseline=-17]
\fill[blue!50] (0,0) rectangle (1,-1);
\draw (0,0) grid (2,-2);
\end{tikzpicture}
\longrightarrow
\begin{tikzpicture}[scale=0.5,baseline=-17]
\fill[blue!50] (1,-1) rectangle (2,-2);
\draw (0,0) grid (2,-2);
\end{tikzpicture}\,.
\]
Define the set of \defn{excited Young diagrams} $\EYD(\lambda/\mu)$ to be all diagrams obtained from $\mu$ using a sequence of elementary excitations such that the resulting diagram is contained inside $\lambda$.

\section{Hook-content formula using excited Young diagrams}

Let $[n]_q! = [n]_q [n-1]_q \dotsm [1]_q$ denote the $q$-factorial.
We define
\[
f^{\lambda/\mu}_q := [\abs{\lambda/\mu}]_q! \sum_{D \in \EYD(\lambda/\mu)} \prod_{d \in \lambda \setminus D} \frac{1}{[\hook(d)]_q}
\]
as the natural $q$-analog of $f^{\lambda/\mu}$.
Note that $\lim_{q\to1} f^{\lambda/\mu}_q = f^{\lambda/\mu}$ by Equation~\eqref{eq:hook_length_formula}.

\begin{thm}
\label{thm:q_hook_content_EYD}
Let $\mu \subseteq \lambda$.
We have
\[
H_{\lambda/\mu}(n; q) := [\abs{\lambda/\mu}]_q! \sum_{D \in \EYD(\lambda/\mu)} \prod_{d \in \lambda \setminus D} \frac{1-q^{n+\content(d)}}{1-q^{\hook(d)}} = f^{\lambda/\mu}_q \prod_{d \in \lambda/\mu} [n+\content(d)]_q.
\]
\end{thm}

\begin{proof}
We first note that
\[
C_{\lambda/\mu}(q) := \prod_{d \in \lambda \setminus D} [n + c(d)]_q
\]
does not depend on the choice of excited Young diagram $D \in \EYD(\lambda/\mu)$ as an elementary excitation moves a box along a diagonal $j-i$, which does not change its content.
Thus, we take $C_{\lambda/\mu}(q)$ to be with $D = \mu$.
Hence, we have
\begin{align*}
H_{\lambda/\mu}(n; q) & = [\abs{\lambda/\mu}]_q! \sum_{D \in \EYD(\lambda/\mu)} \prod_{d \in \lambda \setminus D} \frac{1-q^{n+\content(d)}}{1-q^{\hook(d)}}
\\ & = [\abs{\lambda/\mu}]_q! \sum_{D \in \EYD(\lambda/\mu)} \prod_{d \in \lambda \setminus D} \frac{[n+\content(d)]_q}{[\hook(d)]_q}
\\ & = C_{\lambda/\mu}(q) [\abs{\lambda/\mu}]_q! \sum_{D \in \EYD(\lambda/\mu)} \prod_{d \in \lambda \setminus D} \frac{1}{[\hook(d)]_q}
= C_{\lambda/\mu}(q) f^{\lambda/\mu}_q
\end{align*}
as desired.
\end{proof}

As a special case of Theorem~\ref{thm:q_hook_content_EYD} when $\mu = \emptyset$, Equation~\eqref{eq:q_hook_content} implies that
\begin{equation}
\label{eq:principal_spec}
s_{\lambda}(1, q, \dotsc, q^{n-1}, 0, 0, \ldots) = q^{b(\lambda)} \frac{H_{\lambda}(n; q)}{[\abs{\lambda}]_q!}.
\end{equation}

\begin{cor}
\label{cor:hook_content_EYD}
Let $\mu \subseteq \lambda$. Then we have
\[
H_{\lambda/\mu}(n; 1) = \abs{\lambda/\mu}! \sum_{D \in \EYD(\lambda/\mu)} \prod_{d \in \lambda \setminus D} \frac{n+\content(d)}{\hook(d)} = f^{\lambda/\mu} \prod_{d \in \lambda/\mu} n + \content(d).
\]
\end{cor}

\begin{proof}
This follows from Theorem~\ref{thm:q_hook_content_EYD} by taking the limit $q \to 1$ with applying L'H\^opital's rule and Naruse's hook-length formula (Equation~\eqref{eq:hook_length_formula}).
\end{proof}

We note that we could have proven Corollary~\ref{cor:hook_content_EYD} directly using a similar argument to Theorem~\ref{thm:q_hook_content_EYD} and Naruse's hook-length formula.
Furthermore, Corollary~\ref{cor:hook_content_EYD} is equivalent to Naruse's hook-length formula.
To simplify our notation, we write $H_{\lambda/\mu}(n) := H_{\lambda/\mu}(n;1)$.

\begin{cor}
Assume Corollary~\ref{cor:hook_content_EYD} holds, then we have
\[
\lim_{n \to \infty} \frac{H_{\lambda/\mu}(n)}{n^{\abs{\lambda/\mu}}} = f^{\lambda/\mu}.
\]
\end{cor}

\begin{proof}
Note that $(n + c(d)) / n \to 1$ as $n \to \infty$, and the claim follows from Corollary~\ref{cor:hook_content_EYD} and the degree of $H_{\lambda/\mu}(n)$ (which is a polynomial in $n$) is $\abs{\lambda/\mu}$.
\end{proof}

To obtain the classical hook-content formula for $\lambda$ and $\mu = \emptyset$, we must divide $H_{\lambda/\mu}(n)$ by $\abs{\lambda}!$ as in Equation~\eqref{eq:principal_spec}.
Therefore, we define the polynomial
\[
\overline{H}_{\lambda/\mu}(n) := \frac{H_{\lambda/\mu}(n)}{\abs{\lambda/\mu}!},
\]
and note that $\overline{H}_{\lambda}(n) = \abs{\SST^n(\lambda)}$ by the hook-content formula.

\begin{ex}
The excited Young diagrams $\EYD(3321/21)$ are
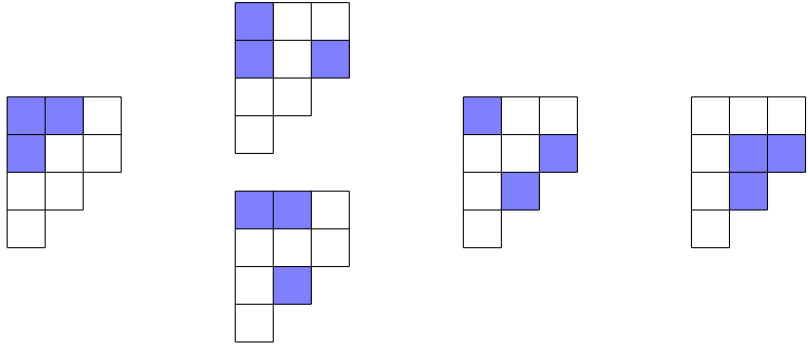
\[
\begin{tikzpicture}[scale=0.5,baseline=-17]
\fill[blue!50] (0,0) rectangle (2,-1);
\fill[blue!50] (0,-1) rectangle (1,-2);
\foreach \y/\ell in {0/3,1/3,2/3,3/2,4/1}
   \draw (0,-\y) -- (\ell,-\y);
\foreach \x/\h in {0/4,1/4,2/3,3/2}
   \draw (\x,0) -- (\x,-\h);
\begin{scope}[xshift=6cm,yshift=2.5cm]
\fill[blue!50] (0,0) rectangle (1,-1);
\fill[blue!50] (2,-1) rectangle (3,-2);
\fill[blue!50] (0,-1) rectangle (1,-2);
\foreach \y/\ell in {0/3,1/3,2/3,3/2,4/1}
   \draw (0,-\y) -- (\ell,-\y);
\foreach \x/\h in {0/4,1/4,2/3,3/2}
   \draw (\x,0) -- (\x,-\h);
\end{scope}
\begin{scope}[xshift=6cm,yshift=-2.5cm]
\fill[blue!50] (0,0) rectangle (2,-1);
\fill[blue!50] (1,-2) rectangle (2,-3);
\foreach \y/\ell in {0/3,1/3,2/3,3/2,4/1}
   \draw (0,-\y) -- (\ell,-\y);
\foreach \x/\h in {0/4,1/4,2/3,3/2}
   \draw (\x,0) -- (\x,-\h);
\end{scope}
\begin{scope}[xshift=12cm]
\fill[blue!50] (0,0) rectangle (1,-1);
\fill[blue!50] (2,-1) rectangle (3,-2);
\fill[blue!50] (1,-2) rectangle (2,-3);
\foreach \y/\ell in {0/3,1/3,2/3,3/2,4/1}
   \draw (0,-\y) -- (\ell,-\y);
\foreach \x/\h in {0/4,1/4,2/3,3/2}
   \draw (\x,0) -- (\x,-\h);
\end{scope}
\begin{scope}[xshift=18cm]
\fill[blue!50] (1,-1) rectangle (3,-2);
\fill[blue!50] (1,-2) rectangle (2,-3);
\foreach \y/\ell in {0/3,1/3,2/3,3/2,4/1}
   \draw (0,-\y) -- (\ell,-\y);
\foreach \x/\h in {0/4,1/4,2/3,3/2}
   \draw (\x,0) -- (\x,-\h);
\end{scope}
\end{tikzpicture}
\]
First, we compute
\begin{equation}
\label{eq:q_hook_example}
f^{3321/21}_q = q^{10} + q^9 + 3q^8 + 6q^7 + 8q^6 + 8q^5 + 9q^4 + 10q^3 +5q^2 + 4q + 5.
\end{equation}
Completing the computation and factoring the result, we see that
\[
H_{3321/21}(n; q) = f^{3321/21}_q[n-3]_q [n-2]_q [n-1]_q [n]_q [n+1]_q [n+2]_q.
\]
We remark that $f^{3321/21}_q = H_{3321/21}(4; q) / [6]_q!$.
By taking $q \to 1$, we obtain
\[
\overline{H}_{3321/21}(n) = \frac{61}{720} (n-3) (n-2) (n-1) n (n+1) (n+2).
\]
as $f^{3321/21} = 61$.
\end{ex}

\begin{ex}
There are five excited diagrams of type $(553, 321)$:
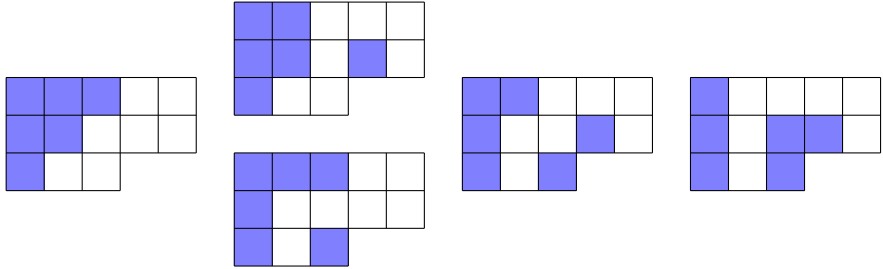
\[
\begin{tikzpicture}[scale=0.5,baseline=-17]
\fill[blue!50] (0,0) rectangle (3,-1);
\fill[blue!50] (0,-1) rectangle (2,-2);
\fill[blue!50] (0,-2) rectangle (1,-3);
\foreach \y/\ell in {0/5,1/5,2/5,3/3}
   \draw (0,-\y) -- (\ell,-\y);
\foreach \x/\h in {0/3,1/3,2/3,3/3,4/2,5/2}
   \draw (\x,0) -- (\x,-\h);
\begin{scope}[xshift=6cm,yshift=2cm]
\fill[blue!50] (0,0) rectangle (2,-1);
\fill[blue!50] (0,-1) rectangle (2,-2);
\fill[blue!50] (0,-2) rectangle (1,-3);
\fill[blue!50] (3,-1) rectangle (4,-2);
\foreach \y/\ell in {0/5,1/5,2/5,3/3}
   \draw (0,-\y) -- (\ell,-\y);
\foreach \x/\h in {0/3,1/3,2/3,3/3,4/2,5/2}
   \draw (\x,0) -- (\x,-\h);
\end{scope}
\begin{scope}[xshift=6cm,yshift=-2cm]
\fill[blue!50] (0,0) rectangle (3,-1);
\fill[blue!50] (0,-1) rectangle (1,-3);
\fill[blue!50] (2,-2) rectangle (3,-3);
\foreach \y/\ell in {0/5,1/5,2/5,3/3}
   \draw (0,-\y) -- (\ell,-\y);
\foreach \x/\h in {0/3,1/3,2/3,3/3,4/2,5/2}
   \draw (\x,0) -- (\x,-\h);
\end{scope}
\begin{scope}[xshift=12cm]
\fill[blue!50] (0,0) rectangle (2,-1);
\fill[blue!50] (0,-1) rectangle (1,-3);
\fill[blue!50] (2,-2) rectangle (3,-3);
\fill[blue!50] (3,-1) rectangle (4,-2);
\foreach \y/\ell in {0/5,1/5,2/5,3/3}
   \draw (0,-\y) -- (\ell,-\y);
\foreach \x/\h in {0/3,1/3,2/3,3/3,4/2,5/2}
   \draw (\x,0) -- (\x,-\h);
\end{scope}
\begin{scope}[xshift=18cm]
\fill[blue!50] (0,0) rectangle (1,-3);
\fill[blue!50] (2,-2) rectangle (3,-3);
\fill[blue!50] (2,-1) rectangle (4,-2);
\foreach \y/\ell in {0/5,1/5,2/5,3/3}
   \draw (0,-\y) -- (\ell,-\y);
\foreach \x/\h in {0/3,1/3,2/3,3/3,4/2,5/2}
   \draw (\x,0) -- (\x,-\h);
\end{scope}
\end{tikzpicture}
\]
which yields the $q$-standard tableau number of
\begin{equation}
\label{eq:q_hook_len_rational}
f^{553/321}_q = 
\frac{(q^6 + q^5 + q^4 + q^3 + q^2 + q + 1) \cdot a(q)}{(q+1) \cdot (q^4 + q^3 + q^2 + q + 1)},
\end{equation}
where
\begin{align*}
a(q) & = q^{12} + 2q^{11} + 4q^{10} + 7q^9 + 12q^8 + 14q^7
\\ & \hspace{20pt} + 17q^6 + 18q^5 + 18q^4 + 14q^3 + 11q^2 + 7q + 5,
\end{align*}
and a hook-content formula (and $q \to 1$ version) of
\begin{align*}
H_{553/321}(n; q) & = f^{553/321}_q [n-1]_q [n]_q [n+1]_q [n+2]_q [n+3]_q^2 [n+4]_q,
\\ \overline{H}_{553/321}(n) & = \frac{91}{5040} (n-1) n (n+1) (n+2) (n+3)^2 (n+4).
\end{align*}
\end{ex}


It is not obvious that $\overline{H}_{\lambda/\mu}(n)$ is an integer for all integers $n \geq \ell$, where $\ell$ is the length of $\lambda$.
However, we have verified this in numerous cases and have the following conjecture.

\begin{conj}
\label{conj:positive_hook_content}
Let $\lambda = (\lambda_1, \lambda_2, \dotsc, \lambda_{\ell})$ be a partition.
Let $n \geq \ell$ be an integer.
Then $\overline{H}_{\lambda/\mu}(n) \in \ZZ_{\geq0}$.
\end{conj}

Thus, if Conjecture~\ref{conj:positive_hook_content} is true, a natural question to ask is what does $\overline{H}_{\lambda/\mu}(n)$ count?
A first guess would likely be semistandard skew tableaux of shape $\lambda/\mu$ and maximum entry $n$, but this is not the case.
Indeed, we have $\overline{H}_{3321/21}(4) = 61$, but there are $204$ semistandard skew tableaux of shape $3321/21$ and maximum entry $4$.
Therefore, we suggest the following problem.

\begin{prob}
Assuming Conjecture~\ref{conj:positive_hook_content}, determine what objects count $\overline{H}_{\lambda/\mu}(n)$.
\end{prob}

We note that the principal specialization $s_{\lambda/\mu}(1, q, \dotsc, q^{n-1}, 0, \ldots)$ was considered in~\cite[Sec.~8]{MPP18}.
Yet this cannot be related to our $q$-hook-content formula as they have different $q \to 1$ limits as noted above.

We note that $f^{\lambda/\mu}_q$ (and hence $H_{\lambda/\mu}(n; q) / [\abs{\lambda/\mu}]_q$ for a fixed integer $n \in \ZZ_{>0}$) is not symmetric nor unimodal as seen in Equation~\eqref{eq:q_hook_example}.
In fact, $f^{\lambda/\mu}_q$ is not always polynomial by Equation~\eqref{eq:q_hook_len_rational} in contrast to Conjecture~\ref{conj:positive_hook_content}.
Furthermore, even when $f^{\lambda/\mu}_q \in \ZZ_{\geq 0}[q]$, the value $H_{\lambda/\mu}(n; q) / [\abs{\lambda/\mu}]_q!$ is not always a polynomial for a fixed integer $n \geq \ell$:
\[
\frac{H_{3322/21}(4;q)}{[7]_q!} = \frac{f(q)}{q^4 + q^3 + q^2 + q + 1},
\]
where
\[
f(q) = q^{12} + 2q^{11} + 4q^{10} + 7q^9 + 12q^8 + 14q^7 + 17q^6 + 18q^5 + 18q^4 + 14q^3 + 11q^2 + 7q + 5.
\]
Note also that $f(q)$ is an irreducible polynomial over $\QQ$.
Yet, we do have the following conjectures based on experimental evidence.

\begin{conj}
\label{conj:q_hook_length_positive_rational}
Let $\mu \subseteq \lambda$ be partitions.
We have $f^{\lambda/\mu}_q = a(q) / b(q)$, where $a,b \in \ZZ_{\geq0}[q]$ such that $a(-1) \in \ZZ_{\geq 0}$.
\end{conj}

\begin{conj}
\label{conj:q_hook_content_positive_rational}
Let $\mu \subseteq \lambda$ be partitions.
Fix some integer $n \geq \ell$, where $\ell$ is the length of $\lambda$.
We have $H_{\lambda/\mu}(n; q) / [\abs{\lambda/\mu}]_q! = a(q) / b(q)$, where $a,b \in \ZZ_{\geq0}[q]$ such that $a(-1) \in \ZZ_{\geq 0}$.
\end{conj}

Note that $g$ in both conjectures must be a product of cyclotomic polynomials since the denominator is a product of $q$-integers.
The examples above also suggests the following problems.

\begin{prob}
Determine which partitions $\mu \subseteq \lambda$ such that $f^{\lambda/\mu}_q \in \ZZ_{\geq 0}[q]$ and also for which $n \in \ZZ_{>0}$ such that $H_{\lambda/\mu}(n;q)/[\abs{\lambda/\mu}]_q \in \ZZ_{\geq 0}[q]$.
\end{prob}

\begin{prob}
For which partitions $\mu \subset \lambda$ the all terms in Naruse's hook-length formula and its $q$-analog are integers and in $\ZZ_{\geq 0}[q]$, respectively? 
\end{prob}

\bibliographystyle{alpha}
\bibliography{hooks}{}

\begin{thebibliography}{MPP18}

\bibitem[Ban08]{Bandlow08}
Jason Bandlow.
\newblock An elementary proof of the hook formula.
\newblock {\em Electron. J. Combin.}, 15(1):Research paper 45, 14, 2008.

\bibitem[FRT54]{FRT54}
J.~S. Frame, G.~de~B. Robinson, and R.~M. Thrall.
\newblock The hook graphs of the symmetric groups.
\newblock {\em Canadian J. Math.}, 6:316--324, 1954.

\bibitem[HG76]{HG76}
A.~P. Hillman and R.~M. Grassl.
\newblock Reverse plane partitions and tableau hook numbers.
\newblock {\em J. Combinatorial Theory Ser. A}, 21(2):216--221, 1976.

\bibitem[IN09]{IN09}
Takeshi Ikeda and Hiroshi Naruse.
\newblock Excited {Y}oung diagrams and equivariant {S}chubert calculus.
\newblock {\em Trans. Amer. Math. Soc.}, 361(10):5193--5221, 2009.

\bibitem[Kon18]{Konvalinka17}
Matjaz Konvalinka.
\newblock A bijective proof of the hook-length formula for skew shapes.
\newblock {\em European J. Combin.}, 2018.
\newblock To appear, \arxiv{1703.08414}.

\bibitem[Kre05]{Kreiman05}
Victor Kreiman.
\newblock Schubert classes in the equivariant {$K$}-theory and equivariant
  cohomology of the {G}rassmannian.
\newblock Preprint, \arxiv{math/0512204}, 2005.

\bibitem[MPP18]{MPP18}
Alejandro~H. Morales, Igor Pak, and Greta Panova.
\newblock Hook formulas for skew shapes {I}. {$q$}-analogues and bijections.
\newblock {\em J. Combin. Theory Ser. A}, 154:350--405, 2018.

\bibitem[Nar14]{Naruse14}
Hiroshi Naruse.
\newblock Schubert calculus and hook length formula, 2014.
\newblock Talk slides at {\it 73rd S\'em.\ Lother.\ Combin., Strobl, Austria},
  \url{http://www.emis.de/journals/SLC/wpapers/s73vortrag/naruse.pdf}.

\bibitem[Sag90]{Sagan90}
Bruce~E. Sagan.
\newblock The ubiquitous {Y}oung tableau.
\newblock In {\em Invariant theory and tableaux ({M}inneapolis, {MN}, 1988)},
  volume~19 of {\em IMA Vol. Math. Appl.}, pages 262--298. Springer, New York,
  1990.

\bibitem[Sag19]{sage}
The Sage Developers.
\newblock {\em {S}age {M}athematics {S}oftware ({V}ersion 8.7)}, 2019.
\newblock \url{http://www.sagemath.org}.

\bibitem[SCc08]{combinat}
The {S}age-{C}ombinat community.
\newblock {S}age-{C}ombinat: enhancing {S}age as a toolbox for computer
  exploration in algebraic combinatorics, 2008.
\newblock \url{http://combinat.sagemath.org}.

\bibitem[Sta71]{Stanley71}
Richard~P. Stanley.
\newblock Theory and application of plane partitions. {I}, {II}.
\newblock {\em Studies in Appl. Math.}, 50:167--188; ibid. 50 (1971), 259--279,
  1971.

\bibitem[Sta99]{ECII}
Richard~P. Stanley.
\newblock {\em Enumerative combinatorics. {V}ol. 2}, volume~62 of {\em
  Cambridge Studies in Advanced Mathematics}.
\newblock Cambridge University Press, Cambridge, 1999.
\newblock With a foreword by Gian-Carlo Rota and appendix 1 by Sergey Fomin.

\end{thebibliography}
\end{document}